\documentclass{amsart}
\usepackage{latexsym,amsxtra,amscd,ifthen,amsmath,color, multicol}
\usepackage{amsfonts}
\usepackage{verbatim}
\usepackage{amsmath}
\usepackage{amsthm}
\usepackage{amssymb}
 \usepackage[all,cmtip]{xy}
 \usepackage[normalem]{ulem}

\usepackage{hyperref}

\numberwithin{equation}{section}

\theoremstyle{plain}
\newtheorem{theorem}{Theorem}[section]
\newtheorem{lemma}[theorem]{Lemma}

\newtheorem{proposition}[theorem]{Proposition}

\newtheorem{corollary}[theorem]{Corollary}
\newtheorem{conjecture}[theorem]{Conjecture}

\theoremstyle{definition}
\newtheorem{definition}[theorem]{Definition}

\newtheorem{notation}[theorem]{Notation}

\newtheorem{remark}[theorem]{Remark}
\newtheorem{question}[theorem]{Question}

\makeatletter              
\let\c@equation\c@theorem  
\makeatother

\newcommand{\mb}{\mathbb}
\newcommand{\CC}{\mb{C}}

\newcommand{\LL}{\mb L}
\newcommand{\NN}{\mb N}
\newcommand{\OO}{\mb O}
\newcommand{\PP}{\mb P}
\newcommand{\ZZ}{\mb Z}

\newcommand{\kk}{{\Bbbk}}

\newcommand{\mf}{\mathfrak}

\newcommand{\mc}{\mathcal}

\newcommand{\sE}{\mc E}
\newcommand{\sF}{\mc F}
\newcommand{\sG}{\mc G}
\newcommand{\sI}{\mc I}
\newcommand{\sJ}{\mc J}
\newcommand{\sK}{\mc K}
\newcommand{\sL}{\mc L}
\newcommand{\sM}{\mc M}
\newcommand{\sN}{\mc N}
\newcommand{\sO}{\mc O}

\newcommand{\sT}{\mc T}

\newcommand{\dra}{\dashrightarrow}
\newcommand{\ssm}{\smallsetminus}

\newcommand{\hra}{\hookrightarrow}

\newcommand{\DMO}{\DeclareMathOperator}
\DMO{\Aut}{Aut}
\DMO{\Div}{div}
\DMO{\ev}{ev}
\DMO{\Bs}{Bs}
\DMO{\GKdim}{GKdim}
\DMO{\Pic}{Pic}
\DMO{\Num}{Num}
\DMO{\Cl}{Cl}
\DMO{\CaCl}{CaCl}
\DMO{\im}{im}

\newcommand{\bbar}[1]{\overline{#1}}
\DeclareMathOperator{\shTor}{\mathcal{T}\!\mathit{or}}

\providecommand{\abs}[1]{\lvert#1\rvert}

\newcommand{\beq}{\begin{equation}}
\newcommand{\eeq}{\end{equation}}

\newcommand{\Anum}{A^1_{\rm{Num}}}

\begin{document}

\title[The universal enveloping algebra of the Witt algebra is not noetherian]
{The universal enveloping algebra of the Witt algebra\\ is not noetherian}

\author{Susan J. Sierra and Chelsea Walton}

\address{Sierra: School of Mathematics, The University of Edinburgh, Edinburgh EH9 3JZ, United Kingdom}

\email{s.sierra@ed.ac.uk}

\address{Walton: Department of Mathematics, Massachusetts Institute of Technology, Cambridge, Massachusetts 02139,
USA}

\email{notlaw@math.mit.edu}

\bibliographystyle{alpha}       

\begin{abstract}
This work is prompted by the long standing question of whether it is possible for the universal enveloping algebra of an infinite dimensional Lie algebra to be noetherian. To address this problem, we answer a 23-year-old question of Carolyn Dean and Lance Small; namely, we prove that the universal enveloping algebra of the Witt (or centerless Virasoro) algebra is not noetherian. To show this, we prove our main result: the universal enveloping algebra of the positive part of the Witt algebra is not noetherian. We employ algebro-geometric techniques from the first author's classification of (noncommutative) birationally commutative projective surfaces.

As a consequence of our main result, we also show that the enveloping algebras of many other (infinite dimensional) Lie algebras  are not noetherian.  These Lie algebras include the Virasoro algebra and all infinite dimensional  $\mathbb{Z}$-graded simple Lie algebras of polynomial growth.

\end{abstract}



\maketitle


\setcounter{section}{-1}


\section{Introduction}

To begin, we take $\kk$ to be a  field of characteristic 0 and we let an unadorned $\otimes$ mean $\otimes_{\kk}$. (If $X$ is a scheme, and $\mc{A}$, $\mc{B}$ are quasicoherent sheaves on $X$, we write $\mc{A} \otimes_X \mc{B}$ rather than $\mc{A} \otimes_{\mc{O}_X} \mc{B}$.)
We are motivated by the  well-known question of whether it is possible for the universal enveloping algebra of an infinite dimensional Lie algebra to be noetherian (cf. \cite[page xix]{GW}).  
It is generally thought that the answer to this question should be ``no," and we state this as a conjecture:

\begin{conjecture}  \label{conj:main}
A Lie algebra $L$ is finite dimensional if and only if the universal enveloping algebra $U(L)$ is noetherian.
\end{conjecture}

One direction holds easily. Namely, the universal enveloping algebra of a finite dimensional Lie algebra is noetherian; see \cite[Corollary~1.7.4]{MR}. To address the converse, many have considered the universal enveloping algebra of the infinite dimensional Lie algebra $W$ below.

\begin{definition}\label{def:U(W)} {} [$W$, $U(W)$] 
The {\it Witt (or centerless Virasoro) algebra} $W$ is defined to be the Lie algebra $W$ with basis $\{e_n\}_{n \in \mathbb{Z}}$ and Lie bracket $[e_n, e_m] = (m-n)e_{n+m}$. We take $U(W)$ to be the  universal enveloping algebra of $W$, which  is $\mathbb{Z}$-graded with deg($e_n$) = $n$. 
\end{definition}

Note that $W$ is realized as the  Lie algebra of derivations of $\kk[x,x^{-1}]$, where $e_n = x^{n+1} \frac{d}{dx}$. If $\kk = \CC$, then  $W$ is also  the complexification of the Lie algebra of polynomial vector fields on the circle. Here, $e_n = -i\exp(in \theta) \frac{d}{d \theta}$, where $\theta$ is the angular parameter.

It is well known that $U(W)$ is a domain, has infinite global dimension, and has sub-exponential growth \cite[Section~3]{DS}.  On the other hand, regarding Conjecture~\ref{conj:main}, we have: 

\begin{question} \label{ques:Lance} (C. Dean and L. Small, 1990) Is $U(W)$ noetherian?
\end{question}

We consider the following subalgebra of $U(W)$, which aids in answering Question~\ref{ques:Lance} above.

\begin{definition}  \label{def:U(W_+)} {} [$W_+$, $U(W_+)$]
The {\em  positive (part of the) Witt algebra}  is defined to be the Lie subalgebra $W_+$ of $W$ generated by $\{e_n\}_{n \geq 1}$.
The universal enveloping algebra $U(W_+)$ is then the following
 subalgebra  of $U(W)$:
$$U(W_+) = \frac{\kk \langle e_n ~|~ n \geq 1 \rangle}{\left([e_n, e_m] = (m-n) e_{n+m}\right)},$$ 
which is $\NN$-graded with $\deg e_n = n$.  
\end{definition}

It is  a general fact that if $L'$ is a Lie subalgebra of $L$ and $U(L)$ is noetherian, then $U(L')$ is noetherian; see Lemma~\ref{lem:liesubalg}.  Because of this, many have asked whether $U(W_+)$ is noetherian.  We show that it is not,  thus 
 answering Question~\ref{ques:Lance} as follows.

\begin{theorem} \label{thm:main} The universal enveloping algebra of the positive Witt algebra $U(W_+)$ is neither right or left noetherian. As a consequence, the universal enveloping algebra of the full Witt algebra $U(W)$ is neither right or left noetherian. Thus, Conjecture~\ref{conj:main} holds for both $W$ and $W_+$.
\end{theorem}

The first step in the proof of the theorem is to produce a homomorphic image of $U(W_+)$ that is {\em birationally commutative}:  
that is, an explicit  algebra homomorphism $\rho:  U(W_+)\to K[t; \tau]$, where $K$ is a field and $\tau \in \Aut_\kk(K)$. We then show, using the techniques of \cite{S-surfclass},  that   $\rho(U(W_+))=:R$ is not noetherian.

The homomorphism  $\rho$ was constructed using 
 the truncated point schemes of $U(W_+)$:  roughly speaking,   the geometric objects parameterizing graded $U(W_+)$-modules with Hilbert series $1+s+\dots + s^n$.  However,  the point schemes are not needed for the proof of Theorem~\ref{thm:main}.  
 We will return to the study of these point schemes in future work.

As a consequence of Theorem~\ref{thm:main}, we also show that many other infinite dimensional Lie algebras satisfy Conjecture~\ref{conj:main}.  (See Section~\ref{CONSEQ} for definitions.)

\begin{corollary}\label{cor:main}
Let $L$ be one of the following infinite dimensional Lie algebras:
\begin{enumerate}
\item the Virasoro algebra $V$; or
\item an infinite dimensional $\mathbb{Z}$-graded simple Lie algebra of polynomial growth.
\end{enumerate}
Then the universal enveloping algebra $U(L)$ is not noetherian.  
Moreover, all central factors of $U(V)$ are non-noetherian. 
\end{corollary}

Preliminary results and lemmas pertaining to the  ring $R$ and its associated geometry are provided in Sections~\ref{PRELIMINARY} and~\ref{GEOM}, respectively. We prove Theorem~\ref{thm:main} in Section~\ref{PROOF}. In Section~\ref{GKR}, we show that the Gelfand-Kirillov dimension of $R$ is 3, which is of independent interest. 
We prove Corollary~\ref{cor:main} in Section~\ref{CONSEQ}.

\section{Preliminaries}\label{PRELIMINARY}

The bulk of this paper is devoted to showing that $U(W_+)$ is not noetherian.  
In this section, we calculate explicitly the defining relations of $U(W_+)$ and construct a useful ring homomorphism $\rho$ from $U(W_+)$ to the ring $K[t; \tau]$ defined in Notation~\ref{not1}.

First, let us produce a second presentation of $U(W_+)$ as follows.

\begin{lemma} \label{lem:present}
Recall Definition~\ref{def:U(W_+)}. We have the following isomorphism:
$$U(W_+) \cong
\frac{\kk \langle e_1, e_2 \rangle}
{\left( 
\begin{array}{c}
[e_1,[e_1,[e_1,e_2]]]+6[e_2,[e_2,e_1]],\\
\left[e_1,\left[e_1,\left[e_1,\left[e_1,\left[e_1,e_2\right]\right]\right]\right]\right]
+40[e_2,[e_2,[e_2,e_1]]] 
\end{array} 
\right)}.$$
\end{lemma}

\begin{proof}
The Lie algebra $W_+$ is generated by $e_1$ and $e_2$. 
The proof of  \cite[Theorem~8.3.1]{Ufnar} shows that $W_+$ has one relation in degree 5 and one in degree 7.  
Thus $U(W_+)$ is generated by $e_1$ and $e_2$ and has one relation in degree 5 and one in degree 7.  

Using the relation $[e_n, e_m] = (m-n) e_{n+m}$, consider the following computations:
\[
[e_1,[e_1,[e_1,e_2]]] = [e_1,[e_1,e_3]]=2[e_1,e_4]=6 e_5,\]
\[ [e_2,[e_2,e_1]] = -[e_2,e_3] = -e_5,\]
\[ [e_1,[e_1,[e_1,[e_1,[e_1,e_2]]]]] = 6 [e_1,[e_1, e_5]] = 24[e_1, e_6] = 120 e_7,\]
\[ [e_2,[e_2,[e_2,e_1]]] = -[e_2,e_5] = -3 e_7.\]
Thus, we have the following two equations:
\beq\label{rel5}
e_1^3 e_2 - 3e_1^2e_2 e_1 + 3 e_1 e_2 e_1^2 -e_2 e_1^3 + 6(e_2^2 e_1 -2 e_2 e_1 e_2 + e_1 e_2^2) =0;
\eeq
\beq \label{rel7}
e_1^5 e_2 - 5 e_1^4e_2 e_1 + 10 e_1^3e_2 e_1^2-10 e_1^2 e_2 e_1^3 + 5 e_1 e_2 e_1^4 - e_2 e_1^5 + 40(e_2^3 e_1-3 e_2^2 e_1 e_2 +3 e_2e_1 e_2^2 -e_1 e_2^3)=0.
\eeq
A routine computation verifies that  the left-hand side of 
 \eqref{rel7} does not lie in the two sided ideal generated by the left-hand 
 side of \eqref{rel5}.  
Thus, these are the degree 5 and degree 7 relations that we seek.
\end{proof}

We will use geometric arguments to analyze $U(W_+)$. Let us establish some notation.  
\begin{notation} \label{not1} [$X$, $\tau$, $f$, $f_i$, $P$, $\kk(X)[t; \tau]$]
We denote $\PP^n_\kk$ simply by $\PP^n$.
Let coordinates on $\PP^3$ be $w,x,y,z$.  Let $X= V(xz-y^2) \subset \PP^3$ be the projective cone over $\PP^1$; this is a rational surface whose singular locus is the vertex  $P = [1:0:0:0]$.    Define an automorphism $\tau$ of $ X$ by
\[ \tau([w:x:y:z]) = [w-2x+2z:z:-y-2z:x+4y+4z].\]
Note that on $X$ we have $z(x+4y+4z)-(-y-2z)^2= xz-y^2=0$,
so $\tau$ is well-defined.
Since the matrix defining $\tau$ is invertible, $\tau$ is an automorphism.  
The automorphism $\tau$ acts on $\kk(X)$ by pullback; by abuse of notation, we denote this pullback action by $\tau$ as well, so that $g^\tau = \tau^* g = g \circ \tau$ for $g \in \kk(X)$.  
We will work in the ring $\kk(X)[t; \tau]$, where $tg = g^\tau t$ for all $g \in \kk(X)$. 

Let 
\[ f= \frac{w+12x+22y+8z}{12x+6y},\]
considered as a rational function in $\kk(X)$; equivalently,  $f: X \dashrightarrow \mathbb{P}^1$ is a rational map.  
For $i \in \ZZ$, let $f_i$ denote $f^{\tau^i} = f \circ \tau^i$.
\end{notation}

In the next result, we construct a homomorphism $\rho$ from $U(W_+)$ to $\kk(X)[t; \tau]$.  

\begin{proposition}\label{prop:rho}
There is a graded algebra homomorphism 
$\rho: U(W_+) \to \kk(X)[t; \tau] $
defined by
$\rho(e_1) = t$ and $\rho(e_2) = ft^2.$
\end{proposition}
\begin{proof}
By Lemma~\ref{lem:present}, we must show that $t$ and $ft^2$ satisfy the equations \eqref{rel5}, \eqref{rel7}.     
We first observe that $\rho$ maps the monomial $e_1^i e_2 e_1^j$ to $t^i ft^{j+2} = f_it^{i+j+2}$.  Similarly, we have
\[ 
\begin{array}{lll} 
e_2^2 e_1 \mapsto f_0 f_2 t^5, &\quad e_2 e_1 e_2 \mapsto f_0 f_3 t^5, &\quad e_1 e_2^2 \mapsto f_1f_3 t^5,\\
e_2^3 e_1 \mapsto f_0 f_2 f_4 t^7, &\quad e_2^2 e_1 e_2 \mapsto f_0 f_2 f_5 t^7, &\quad e_2 e_1 e_2^2 \mapsto f_0 f_3 f_5 t^7, \quad e_1 e_2^3 \mapsto f_1f_3 f_5 t^7.
\end{array}
\]
To verify the relations \eqref{rel5}, \eqref{rel7}, we must therefore check that the following equations hold:
\[
\begin{array}{l}
 f_3 -3 f_2 + 3f_1-f_0 + 6(f_0 f_2 - 2f_0f_3+f_1f_3)=0;\\
f_5-5f_4+10f_3-10f_2+5f_1-f_0+40(f_0f_2f_4 - 3 f_0 f_2 f_5 + 3 f_0 f_3f_5-f_1f_3f_5)=0.
\end{array}
\]
This is a straightforward computation, although it is best done by computer.  
See Routine~A.1 in the appendix for the Macaulay2 calculations.
\end{proof}

Consider the following notation.

\begin{notation} \label{not:R} [$R$] Let $R$ denote the image of $U(W_+)$ under the map $\rho$ of Proposition~\ref{prop:rho}.    
\end{notation}

We will show that $U(W_+)$ is not noetherian by showing that $R$ is not noetherian.

To end the section, we give two useful  technical results.

\begin{lemma}\label{lem:liesubalg}
Let $L$ be a Lie algebra, and let $L'$ be a Lie subalgebra.  If $U(L)$ is noetherian, then $U(L')$ is also noetherian.
\end{lemma}
\begin{proof} 
Any enveloping algebra is isomorphic to its opposite ring, by \cite[Proposition~2.2.17]{Dixmier}.  
Thus, the noetherian property for enveloping algebras is left-right symmetric, and it suffices to show that if $U(L)$ is  noetherian, then $U(L')$ is left noetherian.  
Let $\{e_i\}_{i \in I'}$ be a basis of $L'$, where $I'$ is some set of indices, and extend to a basis $\{e_i\}_{i \in I}$ for $L$, where $I \supseteq I'$.  
Fix an ordering on $I$ so that if $i \in I'$ and $j \not\in I'$, then $j < i$.  
It follows from  the Poincar\'e-Birkoff-Witt theorem that  $U(L)$ has a basis $\{e_{k_1}^{\alpha_{1}}e_{k_2}^{\alpha_{2}} \cdots e_{k_m}^{\alpha_{m}}\}$ for $\alpha_i \in \mathbb{Z}_{\geq 1}$ and $k_i \in I$ with $k_1 < \dots < k_m$, and similarly for $U(L')$. Thus, it is clear that $U(L)$ is a free right $U(L')$-module with basis $\{ e_{j_1}^{\beta_{1}}e_{j_2}^{\beta_{2}} \cdots e_{j_r}^{\beta_{r}} |\ \beta_i \in \mathbb{Z}_{\geq 1}, j_1 < \dots < j_r \in I \ssm I'\} $ over $U(L')$. This implies that $U(L)$ is right faithfully flat over $U(L')$. Now by \cite[Exercise~17T]{GW},  if $U(L)$ is left noetherian, then $U(L')$ is also left noetherian. 
\end{proof}

\begin{lemma}\label{lem:AZ} 
\cite[Proposition~5.10(1)]{AZ}
If $S = \bigoplus_{n \in \NN} S_n$ is a  right (left) noetherian $\NN$-graded $\kk$-algebra, and $N$ is a positive integer, then the Veronese subalgebra $S^{(N)} = \bigoplus_{n \in \NN} S_{Nn}$ is right (left) noetherian. \qed
\end{lemma}


\section{Geometry on $X$}\label{GEOM}

In this section, we give some geometric results about $X$, about the automorphism $\tau$ (from Notation~\ref{not1}) and about certain sheaves on $X$. We will use the following notation throughout.

\begin{notation} \label{not:tildeX} [$\pi$, $\sigma$]
Consider the  rational map $\pi:  X \dra \PP^1$ defined by
\[ [w:x:y:z] \mapsto \begin{cases} [x:y] & \text{if  $x \neq 0$} \\ [y:z] & \text{if  $z \neq 0$.} \end{cases}\]
If both $x$ and $z$ are nonzero, then $y^2 = xz$ is nonzero, and we have that $[x:y] = [xz:yz]=[y^2:yz] = [y:z]$. So, $\pi$ is well-defined. Also, $x=z=0$ intersects $X$ at $P=[1:0:0:0]$, so the domain of definition of $\pi$ is $X \ssm P$. Let $\sigma:\PP^1_{[u:v]} \to \PP^1_{[u:v]}$ be given by $[u:v] \mapsto [v:-u-2v]$. Since rational maps between irreducible projective varieties are equal if they agree on an open set,  we have that $\pi \tau = \sigma \pi$. (Here, we have equality on the open set: $z \neq 0$.) 
\end{notation}

\begin{notation}\label{not:L_i,D}[$p_i$, $L_i$, $D$] 
For $i \in \ZZ$, let  $p_i = \sigma^{-i}([1:-2])$. Note that $\sigma^{-1}([a:b]) = [-b-2a:a]$. Let $L_i =  \overline{\pi^{-1}(p_i)}$. 
 Each $L_i$ is a line through the vertex $P$ of $X$; for example, $L_0 = V(2x+y,2y+z)\subset X$, and  $L_1 = V(x,y)$. 
We also have $\tau^{-1}(L_i) = L_{i+1}$.

   As is well-known, the lines $L_i$ are Weil divisors on $X$    but are not locally principal at $P$
 (note that $X$ is normal, so it makes sense to talk about Weil divisors).
The divisor class group of the local ring $\sO_{X,P}$ is $\ZZ/2\ZZ$, so any sum $L_i+L_j$ is locally principal.
     (See \cite[Examples~II.6.5.2 and II.6.11.3]{Hartshorne}.)

Let $D$ denote the divisor $V(w+12x+22y+8z)\cap X$ on $X$.
\end{notation}

We will need to consider the locally principal Weil divisors $\Div(g)$ for rational functions $g \in \kk(X)$.  
We begin by computing $\Div (f)$ for $f$ in Notation~\ref{not1}.  
Note that $V(2x+y) \cap X$ is a degree 2 curve in $\PP^3$ that contains $L_0$ and $L_1$.  
Thus, $V(2x+y)\cap X = L_0 \cup L_1$ and we have proven the following result.

\begin{lemma} \label{lem:div(f)} Recall Notations~\ref{not1} and~\ref{not:L_i,D}. We have that  $\Div (f) = D - L_0 - L_1$.  As a consequence, $\Div(f_i) = \tau^{-i}(D) - L_i - L_{i+1}$. 
 \qed
\end{lemma}

As is standard, we identify locally principal Weil divisors on $X$ with Cartier divisors; cf. \cite[Remark~II.6.11.2]{Hartshorne}. 
By \cite[Proposition~II.6.13]{Hartshorne}, for any scheme $V$ there is a natural bijection between  Cartier divisors  and  invertible subsheaves of the sheaf $\sK_V$ of total rings of quotients  of $V$.  
Applied to $X$, this bijection pairs  a locally principal Weil divisor $Z$ with the invertible sheaf $\sO_X(Z)$.
The sheaf $\sO_X(Z)$ is defined as follows.
Let $\{U_j\}$ be an open affine cover of $X$ so that each $Z \cap U_j$ is principal, defined by some $z_j \in \sK_X(U_j)$.
Then $\sO_X(Z)(U_j) = z_j^{-1} \sO_X(U_j)$.  
That is, an element of $\sO_X(Z)(U_j)$ is a rational function that has poles no worse than $Z $ on $U_j$.
A global section of $\sO_X(Z)$ is a rational function $g$ so that for each $U_j$, we have
$g = a_j z_j^{-1}$ for some $a_j \in \sO_X(U_j)$.
Equivalently, we have $g z_j \in \sO_X(U_j)$ for all $j$, or that $\Div(g) + Z$ is effective.  Recall that we write this as $\Div(g) + Z \geq 0$.

\begin{notation}\label{notB} [$\mathbb{L}_n$, $\sL_n$, $B(X, \mathcal{L}, \tau^2)$]
Let $\sL \cong \sO(1)|_X$ be the invertible sheaf $\sO_X(L_0+L_1)$, which we identify with a subsheaf of $\sK_X$ as above.
For any $n\in\NN$, let $\LL_{n} = L_0 + L_1 + \dots + L_{2n-1}$.  Note that this is locally principal.  
For $n \geq 1$, let 
\[\sL_n = \sO_X(\LL_n) = \sL \otimes_X (\tau^2)^* \sL \otimes_X \dots \otimes_X (\tau^{(2n-2)})^* \sL.\]
The graded vector space  
$\bigoplus_{n \in \NN} H^0(X, \sL_n)$
has a natural multiplication, induced from the maps $\sL_n \otimes_X (\tau^{2n})^* \sL_m \cong  \sL_{n+m}$ and the maps
\[ \xymatrix{
H^0(X, \sL_n )\otimes H^0(X, \sL_m) \ar[rr]^{1 \otimes (\tau^{2n})^*} && H^0(X, \sL_n) \otimes H^0(X, (\tau^{2n})^* \sL_m )
\ar[r] & H^0(X, \sL_n \otimes_X (\tau^{2n})^* \sL_m).}
\]
The resulting graded algebra is the {\em twisted homogeneous coordinate ring} $B = B(X, \sL, \tau^2)$ \cite{AV}.
Using an indeterminate $t$ to keep track of the graded pieces of $B$, we can write $B$ as
\[ B= B(X, \sL, \tau^2) = \bigoplus_{n \in \NN} H^0(X, \sL_n) \cdot t^{2n}.\] 
Thus, the inclusions  $\sL_n \subset \sK_X$ induce a natural inclusion $B \subset \kk(X)[t^2; \tau^2]$.
\end{notation}

Our next result is that $R^{(2)}$, the second Veronese  of $R$, is contained in $B$.  

\begin{notation} \label{not:Vn} [$V_n$]
For each $n \in \NN$, let $V_n := R_n t^{-n} \subset \kk(X)$.  
\end{notation}

\begin{lemma}\label{lem:subalg}
  Let $B=B(X, \sL, \tau^2)$, considered as a subalgebra of $\kk(X)[t^2; \tau^2] $ as above.
Then $R^{(2)}  \subseteq   B$.   That is, $V_{2n} \subseteq H^0(X, \sL_n)$ for all $n \in \NN$.  
\end{lemma}
\begin{proof}
We must show that $V_{2n} \subseteq H^0(X, \sL_n)$, where the vector space  $H^0(X, \sL_n)$ consists of all rational functions $g$ so that $\Div (g) + \LL_n \geq 0$ (by the comments before Notation~\ref{notB}).

Now,  $U(W_+)_{2n}$ is spanned by all words in $e_1$ and $e_2$ of degree $2n$. In other words, $U(W_+)_{2n}$ is spanned by $\{   e_1^{j_0} e_2 e_1^{j_1} \cdots e_2 e_1^{j_k} \vert\  j_0, \dots, j_k \geq 0, ~ 2k + \sum_{a=0}^k j_a = 2n \}$.
Therefore by Proposition~\ref{prop:rho}, $V_{2n}$ is spanned by 
\begin{multline*} \left\{ f_{j_0} f_{j_0+j_1+2} \cdots f_{j_0 + \dots + j_{k-1}+2k-2} ~\vert~ j_0, \dots, j_k \geq 0,  ~2k + \textstyle \sum_{a=0}^k j_a = 2n \right\} \\
=\left\{ f_{i_1} f_{i_2} \cdots f_{i_k} \vert \ i_1 \geq 0, i_a \leq i_{a+1}-2  \text{ for } 1 \leq a \leq k-1, i_k \leq 2n-2 \right\}. 
\end{multline*}
 It suffices to show for any such rational function $m = f_{i_1} f_{i_2} \cdots f_{i_k}$ that $\Div (m) +  \LL_n \geq 0$.  

By Lemma~\ref{lem:div(f)}, we have that
\[ \Div (m) = \tau^{-i_1}(D)+\dots + \tau^{-i_k} (D) - (L_{i_1} + L_{i_1+1} + L_{i_2}+L_{i_2+1} + \dots +L_{i_k}+ L_{i_k+1}).\]
Whatever the choice of $i_1, \dots, i_k$, the conditions on the $i_a$ ensure that 
\[ L_{i_1} + L_{i_1+1} + L_{i_2}+L_{i_2+1} + \dots + L_{i_k}+L_{i_k+1} \leq \LL_n,\]
so 
$\Div (m) + \LL_n  \geq 0$ as required.
\end{proof}

From now on, we consider $R^{(2)} \subseteq B$ without comment.  

We introduce some geometric notions attached to $V_n$; c.f. \cite[Definition~1.1.8]{Laz} for further details.

\begin{definition} \label{def:base} {} [Bs($|V|$)]
Let $Y$ be a projective scheme, let $\sM$ be an invertible sheaf on $Y$, and let $V \subseteq H^0(Y, \sM)$ be a nonzero subspace.
Consider the natural {\em evaluation map} $\ev:  H^0(Y, \sM) \otimes \sO_Y \to \sM$.  
We have $\ev(V \otimes \sO_Y) \subseteq \sM$; there is thus an ideal sheaf $\sI$ on $Y$ so that $\ev(V\otimes \sO_Y) = \sI \sM$.
The {\em base locus} of $V$ is the subscheme of $Y$ defined by $\sI$.   It is denoted $\Bs(\abs V)$.   
If $\ev(V \otimes \sO_Y) = \sN$ for some sheaf $\sN$, we say that $V$ {\em generates} $\sN$.  
  \end{definition}

\begin{remark} \label{rem:Bs(|V|)} 
 Suppose that $\sM= \sO_X(Z)$ for some effective locally principal Weil divisor $Z$ on $X$.
Let $g_1, \dots, g_k \in \kk(X)$ be a basis for $V \subseteq H^0(X, \sM)$.
To compute $\ev(V \otimes \sO_X) = \sN \subseteq \sM$, we work locally.  
Write $\Div(g_i) + Z = A_i$, where $A_i$ is effective and locally principal.
On an affine open set $U_j \subseteq X$, the locally principal divisor $A_i$ is defined by some $a_{ij} \in \sO_X(U_j)$,
and we have $g_i = a_{ij}z_j^{-1}$.
Then 
\[ \sN(U_j) ~=~ 
\sum_{i=1}^k g_i \sO_X(U_j)~=~
z_j^{-1} \cdot \left( \sum_{i=1}^k a_{ij}  \sO_X(U_j) \right) 
~\subseteq~ z_j^{-1} \sO_X(U_j)
~=~ \sO_X(Z)(U_j).\]
Notice that the ideal $(a_{1j},\dots, a_{kj})$ of $\sO_X(U_j)$ defines $A_1 \cap \dots \cap A_k \cap U_j$. 
 We see that $\sN = \sI \sM$, where $\sI$ is the defining ideal of $A_1 \cap \dots \cap A_k$, or that
\begin{equation} \label{eq:Bs|V|}
 \Bs(\abs V) = \bigcap_{g \in V} (\Div ( g) + Z).
\end{equation}
 \end{remark}

 Our next task is to consider the base loci of the vector spaces  $V_{2n} \subseteq H^0(X, \sL_n)$.

\begin{notation}\label{not:CrCs}[$C_r$, $C_s$, $r_i$, $s_i$, $\OO(q)$]
 We  define curves
\[ C_r = V(w+4y + 2z )\cap X \quad \text{ and } \quad C_s = V(w+6x + 16y + 8z ) \cap X.\]
Both $C_r$ and $C_s$ are contained in the smooth locus of $X$, since neither contains the vertex $P=[1:0:0:0]$.  
Let $r_0 = [0:1:-2:4] = L_0 \cap C_r$ and let $s_1 = [8:0:0:-1] = L_1 \cap C_s$.  For $i \in \NN$, let $r_{i} = \tau^{-i}(r_0) \in L_{i}$ and let $s_{i} = \tau^{-i+1}(s_1) \in L_{i}$.  

For $q \in X$, let $\OO(q)$ denote the orbit $\{ \tau^n(q) ~|~ n\in \ZZ\}$.
\end{notation}

\begin{lemma}\label{lem:base2} Recall Notations~\ref{not1},~\ref{not:tildeX}, ~\ref{not:Vn}, and~\ref{not:CrCs}. Then we have the following statements.
\begin{enumerate}
\item There is a scheme-theoretic equality: $\Bs(\abs{ V_2}) = \{ r_0, s_1\}$.
\item  The $\tau$-orbits of $r_0$ and $s_1$ are distinct and infinite.  In particular, both $r_0$ and $s_1$ map under $\pi$ to points of infinite $\sigma$-order on $\PP^1$.    Moreover, neither $\OO(r_0)$ nor $\OO(s_1)$ is Zariski-dense.
\end{enumerate}
\end{lemma}
\begin{proof}
(a) 
The elements $t^2=\rho(e_1^2)$ and $ft^2=\rho(e_2)$ are a basis for $R_2$, so the base locus of $V_2$ is the intersection of 
$ \Div(1) + \LL_1 = \LL_1$ and $\Div(f) + \LL_1 = D$ by Lemma~\ref{lem:div(f)} and Remark~\ref{rem:Bs(|V|)}.  
By direct computation, we have that $ L_1 \cap D = [8:0:0:-1] = s_1$ and $ L_0 \cap D = [0:1:-2:4] = r_0$.  
Thus  $\LL_1 \cap D = \{ r_0, s_1\}$, set-theoretically.

Since $\LL_1 \cap D = V(xz-y^2) \cap V(2x+y) \cap V(w+12x+22y+8z)$ is finite, 
by B\'ezout's Theorem \cite[Proposition~8.4]{Fulton}, the scheme-theoretic intersection of $\LL_1$ and $D$ consists of  two points.  
Thus the scheme-theoretic intersection $\LL_1 \cap D$ is  $ \{ r_0, s_1\} $.

\smallskip

\noindent (b) One can easily check that the curves $C_r$ and $C_s$ are $\tau$-invariant.
Thus, $\OO(r_0) \subseteq C_r$ and $\OO(s_1) \subseteq C_s$, and so  neither orbit is dense.

Consider the automorphism $\sigma$ of $\PP^1$ from Notation~\ref{not:tildeX}.  
The matrix {\footnotesize $\begin{pmatrix} 0 & 1 \\ -1 & -2 \end{pmatrix}$}  for $\sigma$ has a unique eigenvector, {\footnotesize $\begin{pmatrix} -1 \\ 1\end{pmatrix}$}.  
That is, $[-1:1]$ is the unique fixed point of $\sigma$.  
Since we are in characteristic 0, all other points of $\PP^1$ have infinite $\sigma$-orbits. 
As neither $r_0$ nor $s_1$ maps to $[-1:1] $ under $\pi$,
 both  $r_0$ and $s_1$ have infinite order under $\tau$.

Finally, we verify that $\OO(r_0)$ and $\OO(s_1)$ are distinct.  
Any point $\OO(r_0) \cap \OO(s_1)$ lies in $C_r \cap C_s$.  
 One can verify that $C_r \cap C_s=[2:1:-1:1]$ (set-theoretically).  
 Two  orbits either do not meet or are equal; it is impossible for $\OO(r_0) \cap \OO(s_1) $ to equal $\{[2:1:-1:1]\}$.  
 Thus, $\OO(r_0) \cap \OO(s_1) = \emptyset$.
\end{proof}

\begin{lemma}\label{lem:allbase}
For any $n \in \NN_{\geq 1}$, there is a scheme-theoretic equality $ \Bs(\abs{V_{2n}}) = \{ r_0, s_{2n-1}\}$.
In particular, $V_{2n} \subseteq H^0(X, \sI_{r_0, s_{2n-1}} \sL_n)$. 
\end{lemma}
\begin{proof}
For $n=1$, this is Lemma~\ref{lem:base2}(a).  Fix $n \geq 2$ and let  $Y= \Bs(\abs{V_{2n}})$.  

We first show that $\{r_0, s_{2n-1}\}  \subseteq Y$.  
Recall from the proof of Lemma~\ref{lem:subalg} that $V_{2n}$ is spanned by monomials of the form 
$\{ f_{i_1} f_{i_2} \dots f_{i_k} ~\vert ~i_1 \geq 0, i_a \leq i_{a+1}-2  \text{ for } 1 \leq a \leq k-1, i_k \leq 2n-2 \}$.  
Let $m =  f_{i_1} f_{i_2} \dots f_{i_k}$ be such a monomial.
  As in the proof of Lemma~\ref{lem:subalg}, we have that
\beq\label{div(m)} 
\Div (m) + \LL_n = \tau^{-i_1}(D)+\dots + \tau^{-i_k} (D) - (L_{i_1} + L_{i_1+1} + \dots +L_{i_k} + L_{i_k+1}) + \LL_n.
\eeq
It follows that if $i_1 \geq 1$ 
, then $\Div (m) + \LL_n \geq L_0$.   
If $i_1=0$ then $\Div(m) + \LL_n \geq D$.  In either case,~$\eqref{div(m)}$ contains the point $r_0 = L_0 \cap D$.
Likewise, if $i_k < 2n-2$ then $\Div (m) + \LL_n \geq L_{2n-1}$; otherwise, $\Div(m) + \LL_n \geq \tau^{-(2n-2)}(D)$.  
In either case, \eqref{div(m)} contains $s_{2n-1}=L_{2n-1} \cap \tau^{-(2n-2)}(D)$.

We now show that $Y \subseteq \{r_0, s_{2n-1}\}$.
First, recall \eqref{eq:Bs|V|} and note that $Y \subseteq \LL_n = \Div(1) + \LL_n$.  
Also, 
\[Y \subseteq \Div(f_0f_2 \cdots f_{2n-2})+ \LL_n = D + \tau^{-2}(D) + \dots + \tau^{-(2n-2)}(D) \subseteq X \ssm P.\]
Thus, it suffices to  consider the intersection of $Y$ with  $L_i \ssm P$ for each  $0 \leq i \leq 2n-1$.
If $0 \leq i \leq 2n-2$, consider
\beq\label{eqfi} \Div(f_{i}) + \LL_n = L_0 + \dots + L_{i-1} + \tau^{-i}(D) + L_{i+2} + \dots + L_{2n-1}.\eeq 
The proof of  Lemma~\ref{lem:base2}(a) shows that the scheme-theoretic intersection of \eqref{eqfi} with $L_i \ssm P$ is $\{r_i=\tau^{-i}r_0\}$, and with  $L_{i+1} \ssm P$ is  $\{s_{i+1}=\tau^{-i}s_1\}$.
Now for $0 \leq i \leq 2n-3$,  consider
$$\Div(f_{i+1})+\LL_n = L_0+\dots + L_i + \tau^{-(i+1)}(D) + L_{i+3} + \dots + L_{2n-1}.$$  
This meets $L_{i+1}\ssm P$ only at $r_{i+1}$, and $L_{i+2}\ssm P$ only at $s_{i+2}$.  
Again, these intersections are scheme-theoretic.
Thus if  $1 \leq j \leq 2n-2$, then $Y\cap L_j \subseteq \{r_j\} \cap \{s_j\}$;
 this intersection is empty by Lemma~\ref{lem:base2}(b).
Further, we have that $Y \cap L_0 \subseteq \{r_0\}$ and $Y \cap L_{2n-1} \subseteq \{s_{2n-1}\}$. 
 Thus, $Y = \{r_0, s_{2n-1}\}$ with the reduced induced scheme structure.
\end{proof}

It is well-known in the study of subalgebras of twisted homogeneous coordinate rings \cite{R-generic, S-surfprop, S-idealizer} that when such algebras are defined using points whose orbits are not dense, they tend not to be noetherian.  
Since Lemma~\ref{lem:base2} implies that neither $r_0$ or $s_1$ has a dense orbit,   it would be extremely surprising for $R$, or therefore for $U(W_+)$, to be noetherian.


\section{Proof of Theorem~\ref{thm:main}}\label{PROOF}

In this section, we apply the geometric results of the previous section to study the ring $R=\rho(U(W_+))$ of Notation~\ref{not:R}. The aim is to prove Theorem~\ref{thm:main}, 
that $U(W_+)$ is not noetherian, by showing that $R$ is not noetherian.  Now, $R$ is a subalgebra of $\kk(X)[t; \tau]$:  it is thus {\em birationally commutative}, and we will see in Section~\ref{GKR} that $R$ has Gelfand-Kirillov (GK) dimension three.  The aim is then achieved by applying techniques from the classification of birationally commutative graded domains of GK dimension~3 as presented in \cite{S-surfprop, S-surfclass}.  
We give a self-contained presentation here, however, instead of quoting results from those papers.

To show that $R$ is not left noetherian, we show that there is a left module-finite $R^{(2)}$-algebra $T$ that is not left noetherian.  

\begin{notation}\label{not:T}
 [$Y$, $\sM_n$, $\bbar{\tau}$,  $\sT_n$, $T$, $U_n$,  $\alpha$]
Consider the non-reduced curve $Y = 2C_r = V((w+4y + 2z)^2 )\cap X$. Note that $Y \cong \PP^1_{\kk[\epsilon]}$, where $\kk[\epsilon] = \kk[u]/(u^2)$ denotes the dual numbers.
We have $\kk(Y) \cong \kk(s)[\epsilon]\cong \kk(s)[u]/(u^2)$.  

Since $C_r$ is $\tau$-invariant, $Y$ is also $\tau$-invariant; let  $\bbar{\tau} = \tau|_Y$.
Let $\sM = \sL\otimes_X \sO_Y$ and let $\sM_n = \sL_n\otimes_X \sO_Y$.  
For $n \geq 1$, let $\sT_n = \sI_{r_0} \sM_n$; let $\sT_0 = \sO_Y$.  
Let 
\[T ~=~ \bigoplus_{n \geq 0} H^0(Y, \sT_n)t^{2n} ~\subseteq~ B(Y, \sM, \bbar{\tau}^2) ~\subseteq~ \kk(Y)[t^2; \bbar{\tau}^2]. \] 
Note that $T$ is a subalgebra of (in fact, an idealizer in) the twisted homogeneous coordinate ring $B(Y, \sM, \bbar{\tau}^2)$.

Let $n \geq 1$.
From the surjection $\sL_n \to \sM_n$ we obtain a map 
\[\alpha_n:  H^0(X, \sL_n) \to H^0(Y, \sM_n).\]
Consider the exact sequence: $0 \to \sI_{r_0, s_{2n-1}}\sL_n \to \sL_n \to \sO_{r_0} \oplus \sO_{s_{2n-1}} \to 0$. Since $s_{2n-1} \not\in C_r$, 
when we restrict to $Y$ we obtain the exact sequence
\[ \sI_{r_0, s_{2n-1}} \sL_n \otimes_X \sO_Y \to \sM_n \stackrel{\delta}{\to} \sO_{r_0} \to 0.\]
Thus  there are natural surjections 
\begin{equation} \label{eq:natsurjs}
\sI_{r_0, s_{2n-1}} \sL_n \twoheadrightarrow \sI_{r_0, s_{2n-1}} \sL_n \otimes_X \sO_Y \twoheadrightarrow \ker{\delta} = \sT_n.
\end{equation}  
Taking global sections, and recalling from Lemma~\ref{lem:allbase} that $V_{2n} \subseteq H^0(X, \sI_{r_0,s_{2n-1}}\sL_n)$, 
we see that $\alpha_n$ maps $V_{2n} \subseteq H^0(X,  \sL_n)$ to $ H^0(Y, \sT_n)$.    Let $U_n = \alpha_n(V_{2n})$.

We define a map $\alpha:   R^{(2)} \to T$ by mapping $R_{2n} = V_{2n}t^{2n} \to U_n t^{2n} \subseteq T$ via $\alpha_n$, and extending linearly.
\end{notation}

\begin{lemma}\label{lem:Talg}
  The map $\alpha$ defined above is an algebra homomorphism, and thus $T$ is an $R^{(2)}$-bimodule.  
\end{lemma}

\begin{proof}
Since $R \subseteq \kk(X)[t; \tau]$, the multiplication $R_n \otimes R_m \to R_{n+m}$ is given by $1 \otimes \tau^n$.  
Since $\bbar{\tau} = \tau|_Y$, the diagram
\[ \xymatrix{
 R_{2n} \otimes R_{2m} \ar[r]^{1 \otimes \tau^{2n}} \ar[d]_{\alpha \otimes \alpha} & R_{2n+2m} \ar[d]^{\alpha} \\
T_n \otimes T_m \ar[r]_{ 1\otimes \bbar{\tau}^{2n}} & T_{n+m}}
\]
commutes.  
Since $T$ is a subalgebra of $\kk(Y)[t^2;  \bbar{\tau}^2]$, the bottom row of the diagram gives the multiplication on $T$.  
Thus, $\alpha$ is an algebra homomorphism, and $T$ immediately obtains an induced $R^{(2)}$-bimodule structure.
\end{proof}

For the next two proofs, let $\sF^{\tau^i} = (\tau^i)^* \sF$ for a quasicoherent sheaf $\sF$ on $X$.  

The proof of the following result is adapted from the proof of \cite[Lemma~7.4]{S-surfclass}.

\begin{theorem}\label{thm:Tfingen}
The algebra $T$ is a  finitely generated left $R^{(2)}$-module.
\end{theorem}
\begin{proof}
For $n, m \geq 1$, consider the natural maps
\[ \xymatrix{
\sT_n \otimes_Y \sT_m^{\tau^{2n}} \ar[r]^{q_1} & \sT_n \otimes_Y \sM_m^{\tau^{2n}} \ar[r]^{q_2} & \sM_n \otimes_Y \sM_m^{\tau^{2n}}.
}\]
The kernel of $q_1$ is $\shTor^Y_1(\sT_n, \sO_{r_{2n}})$.  This is zero, since $\sT_n$ is locally free at $r_{2n} \neq r_0$, so $q_1$ is injective.  
Likewise, $q_2$ is injective because $\sM_m^{\tau^{2n}}$ is locally free.
We see that $\sT_n \otimes_Y \sT_m^{\tau^{2n}} \cong \im(q_2 q_1) = \sI_{r_0, r_{2n}} \sM_{n+m}$.  

Let $i \geq 1$, and let
\[ \ev_X:  H^0(X, \sL_i) \otimes \sO_X \to \sL_i, \quad\quad
\ev_Y:   H^0(Y, \sM_i) \otimes \sO_Y \to \sM_i\]
denote the two evaluation maps.
Restricting $\ev_X$ to $Y$,
we obtain a commutative diagram
\[
\xymatrix{
H^0(X, \sL_i) \otimes \sO_X \ar[r]^(0.7){\ev_X} \ar[d] & \sL_i \ar[d] \\
H^0(X, \sL_i) \otimes \sO_Y \ar[r]& \sM_i,}\]
where the downward arrows are induced from the surjection $\sO_X \to \sO_Y$.
The bottom map clearly factors through $\ev_Y$, and we thus obtain a commutative diagram 
\beq\label{foo}
\xymatrix{
V_{2i} \otimes \sO_X \ar[r]^{\ev_X} \ar[d] & \sI_{r_0, s_{2i-1}} \sL_i \ar[d] \\
U_i \otimes \sO_Y \ar[r]_{\ev_Y} & \sM_i.}
\eeq
Here the left-hand arrow is the tensor product of $\alpha_i$ with $\sO_X \to \sO_Y$.  

Consider the maps  in \eqref{foo}.  
By Lemma~\ref{lem:allbase}, the top map is surjective,
and the left-hand map is surjective by construction.  
We have seen in the discussion of \eqref{eq:natsurjs} that the image of the right-hand map is $\sT_i$.  
This is thus also the image of the bottom map, and there is  a surjection
\begin{equation} \label{eqT1}
\xymatrix{ U_i \otimes \sO_Y \ar@{->>}[r]^(0.6){\ev_Y} &  \sT_i.}
\end{equation}
In particular, $\sT_i$ is globally generated.

Let $n \geq 1$.  
Tensoring \eqref{eqT1} with $\sT_n^{\tau^{2i}}$, we obtain 
\[
\xymatrix{ U_i \otimes \sT_n^{{\tau}^{2i}} \ar@{->>}[r]^{\beta_{i,n}}&  \sT_i \otimes_Y \sT_n^{{\tau}^{2i}}.}
\]
\smallskip

{\bf Claim 1:} 
For all $i \geq 1$, there is some $d_i \in \NN$ so that the induced map
\[ \nu_{i,n}: \xymatrix{\alpha(R_{2i})\otimes T_n = U_i t^{2i} \otimes T_n \ar[r]^(0.52){1 \otimes {\tau}^{2i}} & U_i \otimes H^0(Y, \sT_n^{{\tau}^{2i}}) t^{2n+2i}
 \ar[rr]^{H^0(\beta_{i,n})}&& H^0(Y, \sT_i \otimes_Y \sT_n^{\tau^{2i}})t^{2n+2i}}\]
is surjective for all $n \geq d_i$.

{\em Proof of Claim 1.}  
Since $\tau^{2i}$ gives an isomorphism between $T_n t^{-2n}= H^0(Y, \sT_n)$ and $H^0(Y, \sT_n^{\tau^{2i}})$, it suffices to show that $H^0(\beta_{i,n})$ is surjective for $n \gg 0$.

Considering \eqref{eqT1}, define $\sK_i$ so that
\[
\xymatrix{ 0 \ar[r] &  \sK_i \ar[r]& U_i \otimes \sO_Y \ar[r]^(0.6){\ev_Y}&  \sT_i \ar[r] &  0}\]
is exact.
At every $q \in Y$, either $\sT_i$ or $\sT_n^{{\tau}^{2i}}$ is locally free, since $i \geq 1$.
Thus $\shTor_1^Y(\sT_i, \sT_n^{{\tau}^{2i}}) = 0$ and so 
\[\xymatrix{
  0 \ar[r] & \sK_i \otimes_Y \sT_n^{{\tau}^{2i}} \ar[r]&  U_i \otimes\sT_n^{{\tau}^{2i}}  \ar[r]^{\beta_{i,n}} &  \sT_i\otimes_Y \sT_n^{{\tau}^{2i}} \ar[r] &  0}\]
is also exact.
There is thus an exact sequence
\[  \xymatrix{ U_i \otimes H^0(Y, \sT_n^{{\tau}^{2i}})\ar[rr]^{H^0(\beta_{i,n})} & &H^0(Y, \sT_i\otimes_Y \sT_n^{{\tau}^{2i}} ) \ar[r] &  H^1(Y, \sK_i \otimes_Y \sT_n^{{\tau}^{2i}})}, \]
and it suffices to show that 
$ H^1(Y, \sK_i \otimes_Y \sT_n^{{\tau}^{2i}})=0$ for $n \gg 0$.

Consider the inclusion $\sI_{L_0}\sM_n \subseteq \sT_n$; the factor is supported at $r_0$ since $L_0 \cap Y  = \{r_0\}$ (set-theoretically).  
Since $C_r$ is contained in the smooth locus of $X$, the Weil divisor $L_0$ on $X$ is locally principal at every point of $C_r$.  
Thus $\sI_{L_0} \sM_n$ is locally free on $Y$, and using the identification of $Y$ with $\PP^1_{\kk[\epsilon]}$ we see that $(\sI_{L_0} \sM_n)^{\tau^{2i}} \cong \sO(2n-1)$.
As $\sO(1)$ is ample on $\PP^1$, we may choose $d_i \geq 1$ so that $H^1(Y, \sK_i \otimes_Y (\sI_{L_0} \sM_n)^{{\tau}^{2i}}) = 0$ for all $n \geq d_i$.

Let $n \geq d_i$.  
Consider the natural map $\gamma: \sK_i \otimes_Y (\sI_{L_0} \sM_n)^{{\tau}^{2i}} \to \sK_i \otimes_Y \sT_n^{{\tau}^{2i}}$.  
The kernel and cokernel of $\gamma$ are supported on $r_0$, and so there is an exact sequence
\[
\xymatrix{ 
0 \ar[r]& \sE \ar[r] &\sK_i \otimes_Y (\sI_{L_0} \sM_n)^{{\tau}^{2i}} \ar[r]^{\hspace{.2in}\gamma} &\sK_i \otimes_Y \sT_n^{{\tau}^{2i}} \ar[r] & \sF \ar[r] & 0,
}\]
where $\sE$ and $\sF$ are 0-dimensional.  
Let $\mc{G}$ be the image of $\gamma$, which yields short exact sequences
\[ 0 \to \sE \to \sK_i \otimes_Y (\sI_{L_0} \sM_n)^{\tau^{2i}} \to \sG \to 0 \quad \quad \text{ and } \quad \quad
0 \to \sG \to \sK_i \otimes_Y \sT_n^{{\tau}^{2i}} \to \sF \to 0. \] 
From the first short exact sequence, we have an exact sequence in cohomology
\[ 
\xymatrix{
H^1(Y, \sK_i \otimes_Y (\sI_{L_0} \sM_n)^{{\tau}^{2i}}) \ar[r]  &H^1(Y, \sG) \ar[r] &H^2(Y, \sE).
}\]
By the choice of $n$, and the fact that $\sE$ has 0-dimensional support, we deduce that $H^1(Y, \sG)=0$.
Moreover, from the second short exact sequence, we have another exact sequence in cohomology
\[ 
\xymatrix{
H^1(Y, \sG) \ar[r] &H^1(Y, \sK_i \otimes_Y \sT_n^{{\tau}^{2i}}) \ar[r] &H^1(Y, \sF).
}\]
Since  $\sF$ has 0-dimensional support,   $H^1(Y, \sK_i \otimes_Y \sT_n^{{\tau}^{2i}})=0$, proving Claim 1.  
\smallskip

{\bf Claim 2:}
Recall that $d_i \in \mathbb{N}$ is the value so that the map $\nu_{i,n}$ from Claim~1 is surjective for all $n \geq d_i$.  
Let $N = \max(d_1+1, d_2+2,3)$.
 We claim that for $n \geq N$, we have that \[ T_n =  R_2 T_{n-1} + R_4 T_{n-2}.\]

{\em Proof of Claim 2.}
By choice of $n$ and Claim 1, we have
a surjection
\[ \xymatrix{
\alpha(R_2) \otimes T_{n-1} \ar@{->>}[r]^(0.4){\nu_{1,n-1}} & H^0(Y, \sT_1 \otimes_Y \sT_{n-1}^{\tau^2}) t^{2n}.}
\]
The image of $\nu_{1,n-1}$ is $R_2 T_{n-1} \subseteq T_n$, so we have $R_2 T_{n-1} = H^0(Y, \sT_1 \otimes_Y \sT_{n-1}^{\tau^2}) t^{2n}$.
Likewise, we have
$R_4 T_{n-2} = H^0(Y, \sT_2 \otimes_Y \sT_{n-2}^{\tau^4}) t^{2n}$.

Now, $\sT_2$ and $\sT_{n-2}$ are globally generated by  \eqref{eqT1}.  
Thus  $\sT_2 \otimes_Y \sT_{n-2}^{{\tau}^4}= \sI_{r_0, r_4} \sM_n$ is also globally generated.  
In particular, there is a section that does not vanish at $r_2$.  
We thus have:
\beq\label{eq1} 
H^0(Y, \sI_{r_0, r_2} \sT_n )t^{2n} = R_2 T_{n-1} \subsetneqq R_2 T_{n-1} + R_4 T_{n-2} \subseteq T_n.
\eeq

There is a short exact sequence $0 \to \sI_{r_0,r_2}\sM_n \to \sI_{r_0}\sM_n \to \sO_{r_2} \to 0$, which yields  the  following exact sequence:
\[ 
\xymatrix{
0 \ar[r] &  H^0(Y, \sT_1 \otimes_Y \sT_{n-1}^{{\tau}^2}) \ar[r]^(0.6){\phi} &  H^0(Y, \sT_n) \ar[r]^{\psi} &  H^0(Y, \sO_{r_2}).
}\]
Now,  $\dim_{\kk} \im(\psi) \leq \dim_{\kk} H^0(Y, \sO_{r_2}) = 1$. Moreover, by~\eqref{eq1} $T_n \neq R_2T_{n-1}$, so $\phi$ is not an isomorphism. 
Hence 
\[ \dim_\kk R_2 T_{n-1} =  \dim_\kk H^0(Y, \sT_1 \otimes_Y \sT_{n-1}^{{\tau}^2}) = \dim_\kk T_n-1.\]
Now by~\eqref{eq1}, we have that  $T_n =  R_2 T_{n-1} + R_4 T_{n-2}$,
 as claimed.
\smallskip

{\bf Claim 3:}  As a left $R^{(2)}$-module, $T$ is finitely generated; in particular,   $T = R^{(2)} (T_{\leq N-1})$, for $N$ defined in Claim~2.
 
{\em Proof of Claim 3.}  This follows from Claim 2 and induction.  
\end{proof}

Now, we establish our main theorem.

\begin{proof}[Proof of Theorem~\ref{thm:main}]
We show that $U(W_+)$ is not noetherian. Hence, $U(W)$ is also not noetherian by Lemma~\ref{lem:liesubalg}.

By \cite[Proposition~2.2.17]{Dixmier}, it suffices to show that $U(W_+)$ is not left noetherian.   
By  Proposition~\ref{prop:rho} and Lemma~\ref{lem:AZ} it suffices to show that $R^{(2)}$ is not left noetherian.
By Theorem~\ref{thm:Tfingen}, it suffices to show that $T$ is not left noetherian.

Let $\sJ$ be the ideal sheaf on $Y$ that defines the reduced curve $C_r$. Note that $\sJ \subseteq \sI_{r_0}\sO_Y$.  
It is easy to see that $\sJ^\tau = \sJ$.
Let 
\[ J = \bigoplus_{n \geq 1} H^0(Y, \sJ \sM_n) t^{2n} \subseteq T.\]
We claim that $J$ is a non-finitely generated left ideal of $T$.  

If $n, k \geq 1$, then 
\[ T_n J_k \subseteq H^0(Y, (\sI_{r_0} \sM_n)\cdot(\sJ \sM_k)^{\tau^{2n}}) t^{2n+2k}
 = H^0(Y, \sI_{r_0} \sJ^{\tau^{2n}} \sM_{n+k}) t^{2n+2k} 
\subseteq H^0(Y, \sJ^{\tau^{2n}} \sM_{n+k}) t^{2n+2k} .
\]
This is $J_{n+k}$, since $\sJ^{\tau^{2n}} = \sJ$.  Thus $J$ is a left ideal of $T$.
(In fact, $J$ is a two-sided ideal, but we do not need the right ideal structure here.)

To see that $J$ is not finitely generated, first observe that $\sJ \sM_n$ is globally generated for all $n \geq 1$.
In fact, if $b_1, \dots, b_\ell \in H^0(Y, \sM_n) \subseteq \kk(Y)= \kk(s)[\epsilon]$ generate $\sM_n$, then
$\epsilon b_1, \dots, \epsilon b_{\ell}$ generate $\sJ \sM_n$.
Also, we have $\sI_{r_0} \sJ \subsetneqq \sJ$ since $\sJ/ \sI_{r_0}\sJ \cong \sJ \otimes_Y \sO_{r_0} \neq 0$.
 As a consequence, $H^0(Y, \sI_{r_0} \sJ \sM_n)  \subsetneqq H^0(Y, \sJ \sM_n)$ for all $n \geq 1$.
Since $T_n J_k \subseteq H^0(Y, \sT_n (\sJ \sM_k)^{{\tau}^{2n}})t^{2n+2k} = H^0(Y,   \sI_{r_0} \sJ \sM_{n+k})t^{2n+2k}$ for $n, k \geq 1$,
 we have 
$T (J_{\leq k}) \neq J$ for any $k \in \NN$.  
Thus, ${}_T J$ is not finitely generated and $T$ is not left noetherian.
\end{proof}


\section{The Gelfand-Kirillov dimension of $R$} \label{GKR}

In this section, we show that the Gelfand-Kirillov (GK) dimension of the ring $R$ (from Notation~\ref{not:R}) is 3. 
This result is not needed for the proof of Theorem~\ref{thm:main}, but is interesting in its own right.

Recall the notation from Section~\ref{GEOM}. We begin by showing that the automorphism $\tau$ acts trivially on $\Pic(X)$, the Picard group of $X$.

\begin{lemma}\label{lem:trivtau}
We have $\Pic(X) \cong \ZZ$.  As a result, $\tau$ acts trivially on $\Pic(X)$.
\end{lemma}
\begin{proof}
Let $\Cl(X)$ denote the group of Weil divisors on $X$ modulo the group of principal Weil divisors (cf. \cite[page~131]{Hartshorne}).
Let $\CaCl(X)$ denote the subgroup of $\Cl(X)$ consisting of the classes of locally principal Weil divisors. Note that by \cite[Exercise~II.6.3]{Hartshorne}, we have $\Cl(X) \cong \Cl(\PP^1) \cong \ZZ$.  
It is an easy exercise (cf. \cite[Exercise~II.6.3]{Hartshorne}) that $\Cl(X)$ is a free abelian group on the generator $L_0$.

Now, $L_0$ is not locally principal, so $L_0 \not \in \CaCl(X)$.  
On the other hand, the Weil divisor $2L_0$ is equal to $V(4x+4y+z)\cap X$
 and  is locally principal.
That is, $\CaCl(X)$ is generated by a hyperplane section of $X$ and is index 2 in $\Cl(X)$.  
In particular, $\CaCl(X) \cong \ZZ$.

By \cite[Proposition~II.6.15]{Hartshorne}, the map $\CaCl(X)  \to \Pic(X)$ given by $Z  \mapsto \sO_X(Z)$
is an isomorphism.
Thus $\Pic(X) \cong \ZZ$, and is generated by $\sO_X(2L_0)$.
It follows that $\tau$ acts trivially on $\Pic(X)$; this can also be seen by direct computation.
\end{proof}

We now use results of Keeler that relate the GK-dimension of a twisted homogeneous coordinate ring to the numeric properties of the automorphism.

\begin{proposition}\label{prop:GKB}
The GK-dimension of $B= B(X, \sL, \tau^2) $ is 3.
\end{proposition}
\begin{proof}
For any extension field $\kk \subseteq \kk'$, we have $\GKdim_{\kk} B = \GKdim_{\kk'} (B\otimes_\kk \kk')$, so it suffices to assume that $\kk$ is algebraically closed.
 
As in \cite{Keeler}, let $\Anum(X)$ be the group of (Cartier) divisors on $X$ modulo numerical equivalence.
The action of $\tau$ on $\Pic(X)$ induces a numeric action on $\Anum(X)$, which is  trivial by Lemma~\ref{lem:trivtau}.
Thus $\tau^2$ also has a  trivial action  on $\Anum(X)$.

Since $\sL$ is ample, by \cite[Theorem~1.2]{Keeler}, $\sL$ is {\em $\tau^2$-ample}. 
(We do not define the term here; informally, the twisted tensor powers $\sL_n$ of $\sL$ have the same positivity properties as the ordinary tensor powers of an ample line bundle.)
By \cite[Theorem~6.1(2)]{Keeler}, therefore, $\GKdim B = \dim X + 1 = 3$.
\end{proof}

Next, we see that $R^{(2)}$ is  a {\em big} subalgebra of $B$.  Informally,  the  rings $R^{(2)}$ and $B$ are birational to each other. 
 More formally, we show that the rational functions in $V_n$ generate $\kk(X)$ as a field for $n \gg 0$.

\begin{lemma}\label{lem:quotient}
Recall Notations~\ref{not1} and \ref{not:Vn}. For every $n \geq 4$, the rational functions in $V_n$ generate $\kk(X)$ as a field.
\end{lemma}
\begin{proof}
For $n \geq 4$, we have $V_n \supseteq V_4 V_1^{n-4}$.  Since $V_1 = \kk$, the latter is $V_4$.
Thus, it suffices to show that $V_4$ generates $\kk(X)$ as a field.
To do this, we compute: 

{\small
\[ \frac{8 f_0 f_2+4f_0-4f_2}{-4f_0f_1-4f_0f_2-6f_0+4f_1+2f_2} = \frac{y}{z}, 
\hspace{.4in}
\frac{f_0f_1-f_0-f_1}{-6f_0^2 f_2 -3 f_0^2-4f_0f_1+5 f_0 f_2 - 3f_0+4 f_1-f_2} = \frac{2x+5y+2z}{w}.\]}

\noindent 
See Routine~A.2 in the appendix for Macaulay2 calculations confirming these computations.

We claim that these two rational functions generate $\kk(X)$.  
To see this, define rational maps
\[\alpha: X  \dra \PP^1 \times \PP^1, \quad [w:x:y:z]  \mapsto [y:z][2x+5y+2z:w],\]
\[ \beta: \PP^1\times \PP^1  \dra X, \quad  [a:b][c:d]  \mapsto [d(2a^2 + 5ab+ 2b^2): c a^2: cab:cb^2].\]
We have that
\begin{align*}
\beta\alpha([w:x:y:z]) &= [ w(2y^2+5yz+2z^2):(2x+5y+2z)y^2:(2x+5y+2z)yz:(2x+5y+2z)z^2]  \\
 &=[w(2xz+5yz+2z^2):(2x+5y+2z)xz:(2x+5y+2z)yz:(2x+5y+2z)z^2],
\end{align*}
using the relation $xz=y^2$. If  $z(2x+5y+2z)\neq 0$, this is $ [w:x:y:z]$.
Also, if $cb(2a^2+5ab+2b^2) \neq 0$, 
then
\[ \alpha\beta([a:b][c:d]) = [cab:cb^2][2ca^2+5cab+2cb^2:d(2a^2+5ab+2b^2)] = [a:b][c:d].\]
Thus $\beta = \alpha^{-1}$ on an open set, so the two are equal as rational maps.   
Therefore, it suffices to show that $\beta^*(y/z) = a/b$ and $\beta^*((2x+5y+2z)/w) = c/d$ generate $\kk(\PP^1\times \PP^1)$, which is obviously true.
\end{proof}

We now use the following result of Rogalski.
\begin{theorem}\label{thm:Dan}
\cite[Theorem~1.1]{R-GK}
Let $K$ be a finitely generated field of transcendence degree 2 over an algebraically closed field $\kk$, and let $\phi \in \Aut_{\kk}(K)$.
Let $S \subseteq K[t; \phi]$ be a locally finite $\NN$-graded subalgebra so that for some $n \in \NN$ and $u \in S_n$, the algebra $\kk[S_n u^{-1}] $ has quotient field $K$.
Then, $\GKdim S \geq 3$. \qed
\end{theorem}

\begin{corollary}\label{cor:dimR}
The GK-dimension of $R$ is 3.
\end{corollary}
\begin{proof}
As in the proof of Proposition~\ref{prop:GKB}, it suffices to assume that $\kk$ is algebraically closed.
Recall we have $R \subseteq \kk(X)[t; \tau] $.
By Lemma~\ref{lem:quotient}, $\kk(X) $ is the quotient field of $\kk[R_4 t^{-4}]$.  
Since $t^4 \in R_4$,   by Theorem~\ref{thm:Dan} $\GKdim R \geq 3$.

Since $R$ is  a domain, as $R^{(2)}$-modules $R$ embeds into $R^{(2)} \oplus R^{(2)}$.  Thus $\GKdim R= \GKdim R^{(2)}$.
Since $R^{(2)} \subseteq B$ by Lemma~\ref{lem:subalg}, we have $ \GKdim R= \GKdim R^{(2)} \leq \GKdim B = 3$ by Proposition~\ref{prop:GKB}. 
\end{proof}

\section{Further consequences} \label{CONSEQ}

As a consequence of Theorem~\ref{thm:main}, we verify in this section that many other infinite dimensional Lie algebras satisfy Conjecture~\ref{conj:main}. The first of these is a central extension of the Witt algebra.

\begin{definition} \label{def:Vir}
The {\it Virasoro algebra} $V$ is defined to be the Lie algebra $V$ with basis $\{e_n\}_{n \in \mathbb{Z}} \cup \{c\}$ and Lie bracket $[e_n,c]= 0$, $[e_n, e_m] = (m-n)e_{n+m} + \frac{c}{12}(m^3 - m) \delta_{n+m,0}$. 
\end{definition}

The Virasoro algebra is ubiquitous in modern physics, particularly in statistical mechanics and string theory. Most notably, its representations play a major role in 2-dimensional conformal field theory (CFT).

\begin{proposition}\label{prop:Virasoro}
The universal enveloping algebra $U(V)$ of $V$ is not noetherian.
\end{proposition}

\begin{proof}
Observe that the factor $U(V)/(c)$ of $U(V)$ is isomorphic to $U(W)$. Then, apply Theorem~\ref{thm:main}.
\end{proof}

Next, we show that  central factors  of $U(V)$ also fail to be noetherian. Namely, let $\lambda \in \kk$ and define $U_\lambda := U(V)/(c-\lambda)$.  The value $\lambda \in \kk$ is known as the {\it central charge}, which has significance in CFT.

\begin{corollary}\label{cor:Ulambda}
The algebra $U_{\lambda}$ is not noetherian for any $\lambda \in \kk$. 
\end{corollary}

\begin{proof}
Let the Poincar\'e-Birkhoff-Witt basis for $U(V)$ be
\beq \label{PBW}
\{ c^\beta e^{\alpha_1}_{k_1} \dots e^{\alpha_n}_{k_n} |\ k_1 < \dots < k_n \in \ZZ, \ \beta, \alpha_1, \dots, \alpha_n \in \ZZ_{\geq 0} \}.
\eeq

Fix $\lambda \in \kk$.  
Note that $W_+ $ embeds in $ V$, and thus we may consider $U(W_+)$ as a subalgebra of $U(V)$.  Linear independence of \eqref{PBW} implies that
$U(W_+) \cap (c - \lambda) = 0$.  Thus we obtain an induced inclusion $U(W_+) \hra U_\lambda$.  

It is easy to see that $U_\lambda$ has a basis of cosets of the form
\[
\{  e^{\alpha_1}_{k_1} \dots e^{\alpha_n}_{k_n} + (c-\lambda) |\ k_1 < \dots < k_n \in \ZZ, ~\alpha_1, \dots, \alpha_n \in \ZZ_{\geq 1} \},
\]
and so  $U_\lambda$ is a faithfully flat right $U(W_+)$-module.  
By \cite[Exercise~17T]{GW}, $U_\lambda$ is not left noetherian.  It is clear that  $U_\lambda$ is  isomorphic to $(U_{-\lambda})^{op}$, and so also fails to be right noetherian.
\end{proof}

Finally, we show that all $\mathbb{Z}$-graded simple Lie algebras of polynomial growth satisfy Conjecture~\ref{conj:main}. Such Lie algebras were classified by Olivier Mathieu  and we repeat his result below.

\begin{theorem} \cite{Mathieu} 
 If $L$ is a $\mathbb{Z}$-graded simple Lie algebra with polynomially bounded growth, then $L$ is one of the following:
\begin{enumerate}
\item a finite dimensional simple Lie algebra $\mathfrak{g}$, or
\item a loop algebra $\mathfrak{g} \otimes_{\kk} \kk[t^{\pm 1}]$, where $\mf g$ is as above, or
\item a Cartan type algebra $\mathbb{W}_n$, $\mathbb{S}_n$, $\mathbb{H}_{2m}$, $\mathbb{K}_{2m+1}$, or
\item the Witt algebra $W$. \qed
\end{enumerate} 
\end{theorem}

To prove that Conjecture~\ref{conj:main} holds for the Lie algebras above, we need the following result.

\begin{lemma} \label{lem:U(L)}
If $L$ is a Lie algebra that contains a Lie subalgebra $L'$ that is either infinite dimensional abelian or isomorphic to $W_+$, then $U(L)$ is not noetherian.
\end{lemma}

\begin{proof}
This follows from Lemma~\ref{lem:liesubalg}, Theorem~\ref{thm:main}, and the fact that 
the universal enveloping algebra of an infinite dimensional abelian Lie algebra is isomorphic to a polynomial ring in infinitely many variables.
\end{proof}

\begin{corollary}\label{cor:simple}
Conjecture~\ref{conj:main} holds for  $\mathbb{Z}$-graded simple Lie algebras with polynomially bounded growth.
\end{corollary}

\begin{proof}
As mentioned in the Introduction, Conjecture~\ref{conj:main} holds for class (a). Since the loop algebras contain an infinite dimensional abelian Lie subalgebra, for instance $g \otimes \kk[t]$ for any $g \in \mf g$, the conjecture  holds for class (b) by Lemma~\ref{lem:U(L)}. 
By Theorem~\ref{thm:main},  class (d) satisfies the conjecture. It remains to verify the result for class (c).

To study class (c), we use the descriptions of the Cartan type algebras as presented in \cite[Section~2.30]{dSL}. Let $n \geq 1$, and let $P_n$ denote the polynomial ring $\kk[x_1, \dots, x_n]$.  Let $D_i$ denote $\frac{d}{dx_i}$. 

The algebra $\mathbb{W}_n$ is the {\it Lie algebra of derivations of $P_n$}, so equal to $P_nD_1 + \dots + P_nD_n$.  For $n > 1$, $\mb W_n \supseteq \mb W_1$.  The algebra $\mathbb{W}_1= P_1 D_1$ contains $W_+$ as a Lie subalgebra, so the result follows from Lemma~\ref{lem:U(L)}. 

The {\it special subalgebra} $\mathbb{S}_n$ of $\mathbb{W}_n$ is given by 
$$\mathbb{S}_n = \left\{ \sum_{j=1}^n p_jD_j ~{\Big |}~ \sum_{j=1}^n D_j(p_j) =0\right\}.$$
Observe that $\mathbb{S}_1$ is the finite dimensional Lie algebra $\kk D_1$, so $U(\mathbb{S}_1)$ is noetherian, and hence $\mathbb{S}_1$ satisifies Conjecture~\ref{conj:main}.
Now for $n \geq 2$, $\mathbb{S}_n$ contains an infinite dimensional abelian Lie subalgebra: $\kk[x_i] D_j$ for $i \neq j$. So, we are done by Lemma~\ref{lem:U(L)}.

Let $n=2m$.  
Every element of the  {\it Hamiltonian subalgebra} $\mathbb{H}_{2m}$ of $\mathbb{W}_n$ can be represented as $$D_H(p) = \sum_{j=1}^{n} \sigma(j) D_j(p) D_{j'}, \quad \text{for } p \in P_n.$$ 
Here, 
\[
\begin{array}{lll}
j' = j+m & \quad \text{and} \quad \sigma(j) = 1, & \quad \text{if } 1 \leq j \leq m,\\
j' = j-m & \quad \text{and} \quad \sigma(j) = -1, & \quad \text{if } m+1 \leq j \leq 2m.
\end{array}
\]
One can check that $\kk[x_j] D_{j'}$ is an infinite dimensional abelian  Lie subalgebra of $\mathbb{H}_{2m}$, so again, we are done by Lemma~\ref{lem:U(L)}.

Lastly, let $n = 2m+1$.  Every element of the  {\it contact subalgebra} $\mathbb{K}_{2m+1}$ of $\mathbb{W}_n$ can be represented as $D_K(p) = \sum_{j=1}^{2m+1} p_j D_j$ for $p \in P_n.$
 Here, 
\[
\begin{array}{rl}
p_j & =x_j D_{2m+1}(p) + \sigma(j') D_{j'}(p) \quad \text{for } j \leq 2m, \\
p_{2m+1} & =2p - \displaystyle \sum_{j=1}^{2m} \sigma(j) x_j p_{j'}.
\end{array}
\]
One can check that for each $j$ with $1 \leq j \leq m$, the infinite dimensional Lie algebra with basis $$\left\{D_K(x_j^r) ~~= ~~r x_j^{r-1}D_{j'} - (r-2)x_j^rD_{2m+1}\right\}_{r \geq 2},$$ 
is an abelian Lie subalgebra of $\mathbb{K}_{2m+1}$. Therefore, by Lemma~\ref{lem:U(L)}, $U(\mathbb{K}_{2m+1})$ is non-noetherian.
\end{proof}

Corollary~\ref{cor:main} follows from combining Proposition~\ref{prop:Virasoro}, Corollary~\ref{cor:Ulambda}, and Corollary~\ref{cor:simple}.


\section{Appendix:  Macaulay2 computations} \label{APPENDIX}

We present the routines needed for the proofs of Proposition~\ref{prop:rho} and  Lemma~\ref{lem:quotient}. Here, we use Macaulay2, version 1.4 \cite{M2}.

\medskip

\noindent {\bf Routine A.1.} For the proof of Proposition~\ref{prop:rho}, we verify that the relations \eqref{rel5} and \eqref{rel7} hold for the image of $e_1$ and $e_2$ under the map $\rho$. First, let us define the coordinate ring of the projective cone $X = V(xz-y^2) \subseteq \mathbb{P}^3$.

\vspace{-.15in}

\begin{multicols}{2}
{\footnotesize
\begin{verbatim}
i1 : K=QQ;
i2 : ringP3=K[x,y,z,w];
i3 : ringX=ringP3/(x*z-y^2);
i4 : use ringX;
\end{verbatim}
}
\end{multicols}

\noindent Define the maps $\tau$, $f = f_0$, and $f_k = (\tau^*)^k f$ for $k = 0, \dots, 5$.

\vspace{-.15in}

\begin{multicols}{2}
{\footnotesize
\begin{verbatim}
i5 : tau=((w,x,y,z)->(w-2*x+2*z,z,-y-2*z,x+4*y+4*z));
i6 : f0=((w,x,y,z)->(12*x+22*y+8*z+w)/(12*x+6*y));
i7 : f1=f0@@tau;
i8 : f2=f1@@tau;
i9  : f3=f2@@tau;
i10 : f4=f3@@tau;
i11 : f5=f4@@tau;
\end{verbatim}
}
\end{multicols}

\noindent Let us now verify that the relations \eqref{rel5} and \eqref{rel7} hold.

\vspace{-.15in}

\begin{multicols}{2}
{\footnotesize
\begin{verbatim}
i12 : Y0=f0(w,x,y,z);
i13 : Y1=f1(w,x,y,z);
i14 : Y2=f2(w,x,y,z);
i15 : Y3=f3(w,x,y,z);
i16 : Y4=f4(w,x,y,z);
i17 : Y5=f5(w,x,y,z);

i18 : Y3-3*Y2+3*Y1-Y0+6*Y0*Y2-12*Y0*Y3+6*Y1*Y3
o18 = 0

i19 : Y5-5*Y4+10*Y3-10*Y2+5*Y1-Y0
      +40*(Y0*Y2*Y4-3*Y0*Y2*Y5+3*Y0*Y3*Y5-Y1*Y3*Y5)
o19 = 0
\end{verbatim}
}
\end{multicols}

\medskip

\noindent {\bf Routine A.2.} For the proof of Lemma~\ref{lem:quotient}, we verify that 
$f_0$, $f_1$, $f_2$ generate the function field of $X$.

\vspace{-.15in}

\begin{multicols}{2}
{\footnotesize
\begin{verbatim}
i20 : fun0=((a,b,c)->4*(2*a*c+a-c));
i21 : fun1=((a,b,c)->-4*a*b-4*a*c-6*a+4*b+2*c);
i22 : fun2=((a,b,c)->(a*b+a-b));
i23 : fun3=((a,b,c)->-(6*a^2*c+3*a^2+4*a*b
                        -5*a*c+3*a-4*b+c));
i24 : fun0(Y0,Y1,Y2)/fun1(Y0,Y1,Y2)
      3y
o24 = --
      3z

o24 : frac(ringX)

i25 : fun2(Y0,Y1,Y2)/fun3(Y0,Y1,Y2)
      12x + 30y + 12z
o25 = ---------------
             6w
o25 : frac(ringX)
\end{verbatim}
}
\end{multicols}


\section*{Acknowledgements}

The authors are grateful to the referee for pointing out several typographical errors and for making suggestions that improved greatly the exposition of this manuscript. We thank Lance Small and Jason Bell for bringing this problem to our attention, thank Dan Rogalski for an extensive review of our manuscript, and thank Tom Lenagan for assistance with references.  We also thank Jason Bell, Ken Brown, and Pavel Etingof  for stimulating discussions after the first version of this manuscript appeared on the arXiv. Part of this work was completed while being hosted by the Mathematical Sciences Research Institute in Spring 2013.  Sierra was supported by the Edinburgh Research Partnership in Engineering and Mathematics; Walton was supported by the National Science Foundation, grant \#DMS-1102548.

\bibliography{Witt_biblio}

\end{document}